\documentclass[12pt]{amsart}

\newtheorem{THM}{Theorem}
\newtheorem{LMA}[THM]{Lemma}
\newtheorem{PROP}[THM]{Proposition}
\newtheorem{CORO}[THM]{Corollary}
\newtheorem{CONJ}[THM]{Conjecture}
\newtheorem{EG}[THM]{Example}

\usepackage{amsmath}
\usepackage{amssymb}
\usepackage{euscript}

\usepackage[pdftex]{graphicx}

\newcommand{\A}{\EuScript{A}}
\newcommand{\B}{\EuScript{B}}
\newcommand{\C}{\EuScript{C}} \newcommand{\CC}{\mathbb{C}}

 \newcommand{\EE}{\mathbb{E}}
\newcommand{\e}{\mathbf{e}}
 \newcommand{\FF}{\mathbb{F}}

\newcommand{\M}{\EuScript{M}}
\newcommand{\N}{\EuScript{N}} \newcommand{\NN}{\mathbb{N}}

 \newcommand{\RR}{\mathbb{R}}
\renewcommand{\S}{\EuScript{S}}

\newcommand{\U}{\EuScript{U}}

 \newcommand{\ZZ}{\mathbb{Z}}

\newcommand{\la}{\langle}
\newcommand{\ra}{\rangle}
\newcommand{\goesto}{\rightarrow}
\newcommand{\drop}{\setminus}
\newcommand{\none}{\varnothing}

\newcommand{\zero}{\boldsymbol{0}}
\newcommand{\supp}{\mathrm{supp}}

\newcommand{\Gram}{\mathrm{Gram}}
\newcommand{\Row}{\mathrm{Row}}
\newcommand{\showon}{\begin{eqnarray*}}
\newcommand{\showoff}{\end{eqnarray*}}
\renewcommand{\cases}[1]{\left\{\begin{array}{rl} #1 \end{array}\right.}

\begin{document}


\title[Lattice of Integer flows]{The lattice of integer flows\\
of a regular matroid}

\author{Yi Su}
\address{Department of Combinatorics and Optimization\\
University of Waterloo\\
Waterloo, Ontario, Canada\ \ N2L 3G1}
\email{\texttt{next456@gmail.com}}
\thanks{Research of Y.S. supported by an NSERC Undergraduate Summer Research 
Award.}

\author{David G. Wagner}
\address{Department of Combinatorics and Optimization\\
University of Waterloo\\
Waterloo, Ontario, Canada\ \ N2L 3G1}
\email{\texttt{dgwagner@math.uwaterloo.ca}}
\thanks{Research of D.G.W. supported by NSERC Discovery Grant OGP0105392.}

\keywords{regular matroid, integral lattice, isometry, reconstruction}
\subjclass{05B35, 52C07}

\date{}

\begin{abstract}
For a finite multigraph $G$, let $\Lambda(G)$ denote the lattice of
integer flows of $G$ -- this is a finitely generated free abelian
group with an integer-valued positive definite bilinear form.
Bacher, de~la~Harpe, and Nagnibeda show that if $G$ and $H$
are $2$-isomorphic graphs then $\Lambda(G)$ and $\Lambda(H)$ are isometric, and
remark that they were unable to find a pair of nonisomorphic $3$-connected
graphs for which the corresponding lattices are isometric.  We explain this by
examining the lattice $\Lambda(\M)$ of integer flows of any regular matroid
$\M$.  Let $\M_\bullet$ be the minor of $\M$ obtained by contracting all
co-loops.  We show that $\Lambda(\M)$ and $\Lambda(\N)$ are isometric if and
only if $\M_\bullet$ and $\N_\bullet$ are isomorphic.
\end{abstract}

\maketitle


\section{Introduction.}

Let $G=(V,E)$ be a (finite undirected connected multi-) graph.
Choose an arbitrary orientation for each edge of $G$, and let $D$
be the corresponding signed incidence matrix:\ $D$ is the $V$-by-$E$ matrix
with entries given by
$$
D_{ve} = \cases{
+1  & \mathrm{if}\ e\ \mathrm{points\ into}\ v\ \mathrm{but\ not\ out},\\
-1 & \mathrm{if}\ e\ \mathrm{points\ out\ of}\ v\ \mathrm{but\ not\ in},\\
0  & \mathrm{otherwise}.}
$$
The matrix $D$ defines a linear transformation $D:\RR^{E}\goesto\RR^{V}$.
The \emph{lattice of integer flows} of $G$ is
$\Lambda(G)=\mathrm{ker}(D)\cap\ZZ^{E}$.  This is a finitely generated
free abelian group with a positive definite integer-valued inner product
$\la\cdot,\cdot\ra$ induced by the Euclidean dot product on $\RR^{E}$.
Of course, the set $\Lambda(G)$ depends on the 
choice of orientations defining the matrix $D$.  Reversing the 
orientation of the edge $e\in E$ results in changing the sign of
the $e$-th coordinate of every element of $\Lambda(G)$.  This changes
neither the group structure nor the inner product structure of the lattice
$(\Lambda(G),+,\la\cdot,\cdot\ra)$.  Thus, the isometry class of this 
lattice is independent of the choice of orientations of the edges, and
depends only on the isomorphism class of $G$.
(An \emph{isometry} of lattices $\Lambda$ and $\Lambda'$ is a bijection
$\psi:\Lambda\goesto\Lambda'$ such that both $\psi$ and $\psi^{-1}$
are abelian group homomorphisms that preserve the bilinear forms on the
lattices.)  Bacher, de~la~Harpe, and
Nagnibeda \cite{BHN} and Biggs \cite{Bi} thoroughly develop the theory of
these lattices and their many interpretations, connections, and analogues.

A natural question of reconstruction arises:\ to what extent can
properties of the graph $G$ be
determined from the isometry class of the lattice $\Lambda(G)$?
Cut-edges of $G$ contribute nothing to $\Lambda(G)$.  Proposition 5 of  
Bacher, de~la~Harpe, and Nagnibeda \cite{BHN} shows that if $G$ and 
$H$ are $2$-isomorphic then $\Lambda(G)$ and $\Lambda(H)$ are 
isometric.   They remark (on page 197) that they were unable to find a pair
of nonisomorphic $3$-connected graphs with isometric lattices
of integer flows.  By Whitney's theorems \cite{Wh} on $2$-isomorphism of
graphs, this suggests that $\Lambda(G)$ and $\Lambda(H)$ are isometric
if and only if the graphic matroids $\M(G)$ and $\M(H)$ are isomorphic
except for co-loops.

This is indeed the case, as follows from Theorem \ref{main} below.
For any matroid $\M$, let $\M_\bullet$ denote the
minor of $\M$ obtained by contracting all co-loops of $\M$.
Let $(\M,E)$ be a regular matroid of rank $r$ on a ground-set $E$.
Then $\M$ has a unique representation (over $\RR)$ as the column-matroid
of a totally unimodular (TU) matrix $M$ (modulo representation equivalence).
The \emph{lattice of integer flows} of $\M$ is
$\Lambda(\M)=\mathrm{ker}(M)\cap\ZZ^{E}$.
This generalizes the construction for graphs, in which case
$M$ is the signed incidence matrix of a connected graph with any row deleted. 
The isometry class of the lattice $\Lambda(\M)$ is independent of the
choice of representing matrix $M$, and depends only on the isomorphism
class of $\M$.  In his foundational work on representability of 
matroids, Tutte worked with a more general concept of 
``chain-groups'' in which the coefficients are from any integral 
domain; see \cite{Tut2}, for example.  The chain-group of
$\M$ with integer coefficients is, in our notation, $\Lambda(\M^{*})$.

\begin{THM} \label{main}
Let $\M$ and $\N$ be regular matroids.  Then
$\Lambda(\M)$ and $\Lambda(\N)$ are isometric if and only
if $\M_\bullet$ and $\N_\bullet$ are isomorphic.
\end{THM}

\begin{CORO}
Let $G$ and $H$ be $3$-connected graphs.  Then
$\Lambda(G)$ and $\Lambda(H)$ are isometric if and only if
$G$ and $H$ are isomorphic.
\end{CORO}
\begin{proof}
Whitney \cite{Wh} shows that $3$-connected graphs $G$ and $H$ are
isomorphic if and only if $\M(G)$ and $\M(H)$ are isomorphic.
Also, since $G$ has no cut-edges $\M(G)$ has no co-loops, so
that $\M(G)_\bullet=\M(G)$, and similarly for $\M(H)$.
The corollary now follows from Theorem \ref{main}.
\end{proof}

Our strategy for proving Theorem \ref{main} is to identify metric properties
of a basis $\B$ of an integral lattice $\Lambda$ that correspond
to $\Lambda$ being the lattice $\Lambda(\M)$ of integer flows of a
regular matroid $\M$, and to $\B$ being a fundamental basis $\B(\M,B)$
of $\Lambda(\M)$ consisting of signed circuits associated with a base $B$
of $\M$. (Since we are dealing both with lattices and with matroids we
use the word ``basis'' for a basis of a lattice, but ``base'' for what is
usually called a basis of a matroid.)

The implementation of this strategy rests on two key ideas.
The first key is a characterization of the signed circuits (or ``simple
flows'') of $\M$ in terms of metric data of the lattice $\Lambda(\M)$, without
reference to their coordinates as vectors in $\ZZ^{E}$.  The second key is to
identify properties of a symmetric integer matrix $A$ which correspond to the
existence of a TU matrix $U$ such that $U^\dagger U = A$:\  we find a necessary
condition on $A$ which we call ``$g$-nonnegativity''; to any $g$-nonnegative
matrix $A$ we associate a certain $\{0,1\}$-matrix  $X(A)$;\ finally, such a $U$
exists if and only if $X(A)$ has a TU signing $U$ such that $U^\dagger U = A$.
An auxiliary result about TU matrices then enables us to complete the proof
of Theorem \ref{main}.

In Section 2 we briefly review some preliminary facts concerning
totally unimodular matrices, regular matroids, and integer flows and cuts.
In Section 3 we develop some facts about signed circuits (or simple flows),
culminating in their characterization by metric data.  In Section 4 we introduce
$g$-nonnegative, $g$-positive, and $g$-feasible matrices, and prove Theorem
\ref{main}.  In Section 5 we conclude with some subsidiary results and examples,
and two conjectures.

We thank the anonymous referees for their constructive criticism, and one
especially for the references and comments regarding Proposition
\ref{cons-decomp} and Conjecture \ref{sixth-root}.

\section{Preliminaries.}

\subsection{Totally unimodular matrices.}

For a matrix $M$ of real numbers, let $M^{\sharp}$ be the matrix of
absolute values of the entries of $M$.
A matrix $U$ with entries in $\ZZ$ is \emph{totally unimodular} (TU) if 
every square submatrix of $U$ has determinant in the set $\{-1,0,+1\}$.
For a $\{0,1\}$-matrix $X$, a \emph{totally unimodular signing} of 
$X$ is a TU matrix $U$ such that $U^{\sharp} = X$.
A matrix $Q$ with entries in $\ZZ$ is \emph{weakly unimodular} (WU) if 
every maximal square submatrix of $Q$ has determinant in the set
$\{-1,0,+1\}$.  Let $I_s$ denote the $s$-by-$s$ identity matrix.
The proof of Lemma \ref{wu2tu} is elementary, and is omitted.

\begin{LMA} \label{wu2tu}
If an $m$-by-$s$ matrix $U$ is WU and contains $I_s$ as a submatrix,
then $U$ is TU.
\end{LMA}

\begin{LMA}[Camion, see Lemma 13.1.6 of \cite{Ox}] \label{camion}
Let $Q$ and $U$ be TU matrices such that $Q^{\sharp} = U^{\sharp}$.
Then $Q$ can be changed into $U$ by multiplying some rows and columns by 
$-1$.
\end{LMA}

Theorem 13.1.3 of \cite{Ox} determines exactly which $\{0,1\}$-matrices have TU 
signings, although we do not need this result until Example \ref{eg1}.

\subsection{Regular matroids.}

A \emph{regular} matroid $(\M,E)$ is the column-matroid of 
some $r$-by-$m$ TU matrix $M$ of rank $r$, represented over the real
field $\RR$.  The columns of $M$ are labelled by the set $E$. Two 
$\FF$-representations $M$ and $M'$ of a matroid are \emph{equivalent} if
there is an $r$-by-$r$ matrix $F$ invertible over $\FF$, an
$E$-by-$E$ $\FF$-weighted permutation matrix $P$, and a field 
automorphism $\sigma:\FF\goesto\FF$ such that
$$
M' = \sigma(FMP).
$$
(The column labels $E$ are also permuted according to $P$.)  Regular matroids
are \emph{uniquely representable} over any field $\FF$, meaning that any two
$\FF$-representations of a regular matroid are equivalent (Corollary 10.1.4
of \cite{Ox}).

Let $\M$ be represented by a TU matrix $M$.
If $B\subseteq E$ is a base of $\M$ then there is a signed permutation
matrix $P$ bringing the labels in $B$ into the first $r$ positions,
and a matrix $F$, invertible over $\ZZ$, such that
$$
FMP = [ I_{r}\ L ]
$$
for some $r$-by-$s$ matrix $L$, where $s=m-r$.  This is a 
representation of $\M$ \emph{coordinatized by $B$}.
Since $M$ is TU, $F$ is invertible over $\ZZ$, and $P$ is a signed permutation
matrix, it follows that $FMP$ is WU.  From Lemma \ref{wu2tu} (and transposition)
it follows that $[I_r\ L]$ is also TU (see also Lemmas 2.2.20 and 2.2.21
of \cite{Ox}).

\subsection{Integer flows, duality, and integer cuts.}

Let $(\M,E)$ be a regular matroid represented by the $r$-by-$m$  
TU matrix $M$.  The \emph{lattice of integer flows of $\M$} is
$$
\Lambda(\M) = \ker(M) \cap \ZZ^{E},
$$
defined up to isometry.  If $B$ is a base of $\M$ and $M=[I_{r}\ L]$ is a 
representation of $\M$ coordinatized by $B$, then the matrix
$$
U = \left[\begin{array}{c} -L\\  I_{s}\end{array}\right]
$$
is such that $MU=O$.  Since $M$ is TU it follows that $U$ is TU, and
since $U$ has rank $s=\dim\ker(M)$, the columns of $U$ form
an ordered basis $\B(\M,B)=\{\beta_{1},\ldots,\beta_{s}\}$ of $\Lambda(\M)$.
This is a \emph{fundamental basis of $\Lambda(\M)$ coordinatized by $B$}.

If $\M$ is represented by $M=[I_{r}\ L]$ then the \emph{dual matroid}
$\M^{*}$ is represented by $U^{\dagger} = [-L^{\dagger}\ I_{s}]$.
If $M$ is TU then $U^\dagger$ is TU.  The \emph{lattice of integer cuts} of
a regular matroid $\M$, represented by $M$, is
$$
\Gamma(\M) = \Row(M)\cap \ZZ^{E},
$$
in which $\Row(M)$ denotes the row-space of $M$.  As a set this depends 
on $M$, but it is well-defined up to isometry.  From the above, it is clear that
$\Lambda(\M)$ and $\Gamma(\M^{*})$ are isometric.  Since the 
definition of $\Lambda(\M)$ implicitly involves matroid duality, some of our
arguments could be simplified slightly by considering $\Gamma(\M)$ instead.
However, to keep things straight we will consider only $\Lambda(\M)$,
except in Subsection 5.1.

Lemma \ref{basis-wu} is a familiar fact, but we prefer to phrase it
just the way we want.
\begin{LMA} \label{basis-wu}
Let $(\M,E)$ be a regular matroid of rank $r$ on a set $E$ of size 
$m$, and let $s=m-r$.  Let $\B$ be any basis for $\Lambda(\M)$, and let
$Q$ be an $E$-by-$s$ matrix with columns given by the elements of $\B$.
Then $Q$ is WU.
\end{LMA}
\begin{proof}
Pick a base $B$ of $\M$ and let $M=[I_{r}\ L]$ represent $\M$ 
coordinatized by $B$.  Then $\B'=\B(\M,B)$ is another basis 
for $\Lambda(\M)$, and any matrix $U$ with these columns is TU.
Since $\B'$ and $\B$ are both bases for the lattice
$\Lambda(\M)$, the change of basis matrix $F$ such that $Q = U F$ has
$\det F=\pm 1$.  Since $U$ is WU it follows that $Q$ is WU.
\end{proof}

If $\B=\{\beta_{1},\ldots,\beta_{s}\}$ is any ordered set of vectors in an
inner-product space, then the \emph{Gram matrix} 
$\Gram(\B)=A=(a_{ij})$ of $\B$ is the $s$-by-$s$ matrix with entries
$a_{ij} = \la \beta_{i},\beta_{j} \ra$ for all $1\leq i,j\leq s$.
Two lattices $\Lambda$ and $\Lambda'$ are isometric if and only if 
they have ordered bases $\B$ and $\B'$, respectively, such that
$\Gram(\B)=\Gram(\B')$.

\section{Simple flows, or signed circuits.}

\subsection{Basic facts.}
Let $(\M,E)$ be a regular matroid represented by a TU matrix $M$, and let
$\Lambda(\M)=\ker(M)\cap\ZZ^{E}$ be its lattice of integer flows (relative to
$M$).  For a column vector $\beta\in\ZZ^{E}$, the \emph{support} of $\beta$
is the subset
$$
\supp(\beta) = \{e\in E:\ \beta(e)\neq 0\}
$$
of $E$.  For $\beta\in\Lambda(\M)$ we have $M\beta = \zero$, so that 
if $\beta\neq\zero$ then $\supp(\beta)$ is a dependent set in $\M$, and hence
contains a circuit (\emph{i.e.} a minimal dependent set) of $\M$.

We require the following familiar facts (and include supporting arguments
as proof sketches).

\begin{LMA} \label{fact-a}
For every $\beta\in\Lambda(\M)$, if $\supp(\beta)$ is a circuit $C$
then $\beta$ spans the subspace of $\ker(M)$ consisting of vectors with
support contained in $C$.
\end{LMA}
\begin{proof}[Proof sketch.]
If there were another linearly independent vector in this subspace then we
could produce a dependent set of $\M$ properly contained in $C$, a
contradiction.
\end{proof}

\begin{LMA} \label{fact-b}
For every $\beta\in\Lambda(\M)$, if $\supp(\beta)$ is a circuit $C$
then all nonzero coordinates of $\beta$ have the same absolute value.
\end{LMA}
\begin{proof}[Proof sketch.]
Let $M_C$ be the submatrix of $M$ supported on columns in $C$,
and write one of the columns of $M_C$ as a linear combination of the others.
This system of linear equations may be redundant -- reducing to an irredundant
subsystem, it can then be solved by Cramer's Rule, and all the
determinants involved are in $\{-1,0,+1\}$ since $M$ is TU.
\end{proof}

An element of $\Lambda(\M)$ is a \emph{simple flow} (or signed
circuit) if it is nonzero, all of its coordinates are in the set $\{-1,0,+1\}$,
and its support is a circuit of $\M$.  Let $\S(\M)$ denote the set of all simple flows in $\Lambda(\M)$.

\begin{LMA} \label{fact-c}
For every circuit $C$ of $\M$ there are exactly two simple flows
$\pm\alpha_C$ with support equal to $C$.
\end{LMA}
\begin{proof}[Proof sketch.]
Since $C$ is dependent, there is a nonzero $\beta\in\ker(M)$ with
$\supp(\beta)\subseteq C$.  Since $C$ is a circuit, $\supp(\beta)=C$.
Now Lemma \ref{fact-c} follows from Lemmas \ref{fact-a} and \ref{fact-b}.
\end{proof}

\begin{LMA} \label{fact-d}
If $B$ is a base of a regular matroid $\M$ then every element of
$\B(\M,B)$ is a simple flow in $\Lambda(\M)$, and hence $\S(\M)$ spans
$\Lambda(\M)$.
\end{LMA}
\begin{proof}[Proof sketch.]
Each element of $\B(\M,B)$ is supported on a circuit, as one easily verifies.
\end{proof}

\begin{LMA} \label{fact-e}
If $M'$ is another matrix that represents $\M$ then
an element of $\Lambda(\M)$ is a simple flow relative to $M'$
if and only if it is a simple flow relative to $M$.
\end{LMA}
\begin{proof}[Proof sketch.]
By uniqueness of representation for regular matroids, $M'=FMP$
as in Subsection 2.2.  Note that $\beta\in\ker(M')\cap\ZZ^E$ corresponds
to $P\beta\in\ker(M)\cap\ZZ^E$, and that $P$ is a signed permutation matrix.
\end{proof}

\subsection{Consistent decompositions.}

By Lemma \ref{fact-d}, each flow in $\Lambda(\M)$ can be expressed as a sum
of simple flows.  For $\beta\in\Lambda(\M)$, a
\emph{consistent decomposition of $\beta$} is a multiset $\A$ of simple
flows such that:\\
(i)\ $\beta = \sum_{\alpha\in\A}\alpha$;\\
(ii)\ for all $\alpha\in\A$, $\supp(\alpha)\subseteq\supp(\beta)$;\\
(iii)\ for all $\alpha\in\A$ and $e\in E$, $\alpha(e)\beta(e)\geq 0$.

Proposition \ref{cons-decomp} is due to Tutte (Theorem 6.2 of \cite{Tut1}
or Theorem 5.43 of \cite{Tut2}).  We reproduce his proof for completeness
and the readers' convenience.

\begin{PROP}[Tutte] \label{cons-decomp}
Let $(\M,E)$ be a regular matroid represented by a WU matrix $M$.
Then every $\beta\in\Lambda(\M)$ has a consistent decomposition $\A$.
\end{PROP}
\begin{proof}
We begin by showing that if $\beta\neq\zero$ then there exists a simple flow
$\alpha$ that \emph{conforms to $\beta$} in the sense that
$\supp(\alpha)\subseteq\supp(\beta)$ and $\alpha(e)\beta(e)>0$ for all
$e\in\supp(\alpha)$.  If there is a counterexample then there is such
a counterexample $\beta$ with $\supp(\beta)$ minimal.  By Lemmas 
\ref{fact-a}, \ref{fact-b}, and \ref{fact-c}, 
$\supp(\beta)$ is not a circuit.  By Lemma \ref{fact-c},  again, there is
a simple flow $\alpha$ with $\supp(\alpha) \subseteq \supp(\beta)$.
Let $e\in\supp(\alpha)$ be such that $|\beta(e)|$ is minimal.
Replacing $\alpha$ by $-\alpha$ if necessary, we may assume that 
$\alpha(e)\beta(e)>0$.  Now $\beta' = \beta-\beta(e)\alpha$ has
$\supp(\beta') \subset \supp(\beta)$.  If $\beta'=\zero$ then
$\alpha$ conforms to $\beta$.  Otherwise, since $\beta$ was a minimal 
counterexample, there is a simple flow $\alpha'$ conforming to $\beta'$.
From the choice of $e\in\supp(\alpha)$ it follows that $\alpha'$
conforms to $\beta$ as well, a contradiction.

The proposition now follows from the base case $\beta=\zero$ (which
has the consistent decomposition $\A=\none$) by an easy induction on
$||\beta||=\sum_{e\in E}|\beta(e)|$.  For the induction step, let
$||\beta||>0$ and let $\alpha$ be a simple flow conforming to $\beta$.
Then $\beta'=\beta-\alpha$ has $||\beta'||<||\beta||$, so by induction
it has a consistent decomposition $\A'$.   Thus, 
$\A=\A\cup\{\alpha\}$ is a consistent decomposition of $\beta$.
\end{proof}

\subsection{Metric characterization.}

\begin{PROP} \label{sf-metric}
Let $(\M,E)$ be a regular matroid represented by a WU matrix $M$.  For
any nonzero $\alpha\in\Lambda(\M)$, the following are equivalent:\\
\textup{(a)}\ the element $\alpha$ is a simple flow of $\Lambda(\M)$
(relative to $M$);\\
\textup{(b)}\ for all nonzero $\beta,\gamma\in\Lambda(\M)$ such that
$\alpha=\beta+\gamma$, $\la \beta,\gamma \ra<0$.
\end{PROP}
\begin{proof}
First, assume that (a) holds, and let $\alpha=\beta+\gamma$ with nonzero 
$\beta,\gamma\in\Lambda(\M)$.   For every $e\in E$ we have 
$\alpha(e)=\beta(e)+\gamma(e)$,  and since $\alpha(e)\in \{-1,0,+1\}$ we must
have $\beta(e)\gamma(e)\leq 0$.  Since the support of $\alpha$ is a circuit 
of $\M$, the supports of $\beta$ and $\gamma$ cannot be disjoint
(since each contains at least one circuit of $\M$).
Therefore  $\la \beta,\gamma \ra<0$, so that (b) holds.

Conversely, assume that (a) fails to hold.  By Proposition \ref{cons-decomp},
$\alpha$ has a consistent decomposition $\A$.
Since $\alpha$ is nonzero, $\A$ is nonempty.  If $|\A|=1$ then $\alpha$ is
a simple flow.  Thus, assume that $|\A|\geq 2$, and let $\beta\in\A$ and
$\gamma=\alpha-\beta$.  Now $\beta$ and $\gamma$ are nonzero, $\alpha=\beta
+\gamma$, and $\beta(e)\gamma(e)\geq 0$ for all $e\in E$, from the
definition of consistent decomposition.  This shows that
$\la \beta,\gamma \ra \geq 0$, so that (b) fails to hold.
\end{proof}

For an arbitrary lattice $\Lambda$ we define the set of 
\emph{simple elements} to be the set $\S(\Lambda)$ of nonzero elements
$\alpha\in\Lambda$ satisfying condition (b) in Proposition \ref{sf-metric}.
Lemma \ref{sf-iso} is immediate.

\begin{LMA} \label{sf-iso}
Let $\psi:\Lambda\goesto \Lambda'$ be an isometry of integer lattices.
Then $\psi$ restricts to a (metric-preserving) bijection from $\S(\Lambda)$
to $\S(\Lambda')$.
\end{LMA}

Lemma \ref{sf-iso} already severely constrains the possibilities for an
isometry $\psi:\Lambda(\M)\goesto \Lambda(\N)$.  How to get an isomorphism
$\phi:\M_\bullet\goesto\N_\bullet$ from this is still not clear,
however.  This is resolved in the next section.

\section{$g$-Feasible matrices, and proof of Theorem 1.}

Let $\B(\M,B)=\{\beta_{1},\ldots,\beta_{s}\}$ be a fundamental basis 
of $\Lambda(\M)$ (coordinatized by some base  $B$  and representing TU matrix $M$).
Let $U$ be the $m$-by-$s$ matrix with $\{\beta_{1},\ldots,\beta_{s}\}$ as columns.
The Gram matrix $A=U^{\dagger}U$ determines the isometry class of 
$\Lambda(\M)$.  The main effort in the proof of Theorem 1 is to reconstruct
(as far as possible) the matrix $U$ from its Gram matrix $A$.  This is
accomplished by Camion's Lemma \ref{camion} and Corollary \ref{tu-xa} 
below.

\subsection{Inclusion/Exclusion.}

Let $\C=\{C_1,...,C_s\}$ be a collection of subsets of a finite set $E$,
and let $[s]=\{1,2,..,s\}$.  For every $S\subseteq[s]$, define
$$
\phi_\C(S) = \left| \bigcap_{i\in S} C_i \right|
\ \ \mathrm{and} \ \
\gamma_\C(S) =  \left| \bigcap_{i\in S} C_i
\drop \bigcup_{j\in [s]\drop S} C_j
\right|.
$$
Here, by convention, $\bigcap\none = E$.  One sees that for every $S\subseteq
[s]$,
$$
\phi_\C(S) = \sum_{S\subseteq S'\subseteq[s]} \gamma_\C(S').
$$
By Inclusion/Exclusion, it follows that for every $S\subseteq [s]$,
$$
\gamma_\C(S) = \sum_{S\subseteq S'\subseteq[s]} (-1)^{|S'\drop S|}\phi_\C(S').
$$
Note that $\gamma_{\C}(S)\geq 0$ for all $S\subseteq[s]$, from the 
definition.

\subsection{$g$-Feasible matrices.}

Let $A=(a_{ij})$ be an $s$-by-$s$ symmetric matrix of integers,
with positive diagonal entries.  The three-element subsets (or
\emph{triples}) $\{h,i,j\}$ of $[s]$ are divided into three types:\
$\{h,i,j\}$ is \emph{positive}, \emph{null}, or \emph{negative} depending on whether
$$
a_{hi}\cdot a_{ij} \cdot a_{jh}
$$
is positive, zero, or negative.  Let $\Delta(A)$ denote the set of negative
triples of $[s]$.  Define a function $f_A: 2^{[s]}\goesto\NN$ as
follows:\ for each $S\subseteq [s]$,
$$
f_A(S) = \cases{
0 & \text{if $S=\none$},\\
0 & \text{if $Y\subseteq S$ for some $Y\in\Delta(A)$},\\
a_{ii} & \text{if $S=\{i\}$},\\
\min \{ |a_{ij}|:\ \{i,j\} \subseteq S \} & \text{otherwise.} }
$$
Define a second function $g_A:2^{[s]}\goesto\ZZ$ by Inclusion/Exclusion:\
for each $S\subseteq [s]$,
$$
g_A(S) = \sum_{S\subseteq S'\subseteq[s]} (-1)^{|S'\drop S|} f_A(S').
$$
The matrix $A$ is \emph{$g$-nonnegative} provided that $g_A(S)\geq 0$ for
all $\none\neq S\subseteq[s]$, and is \emph{$g$-positive} if it
is $g$-nonnegative and such that $g_A(\{i\})>0$ for all $i\in [s]$.  Notice that, since $f_A(\none)=0$,
if $A$ is $g$-positive and $[s]\neq\none$ then
$$
g_A(\none) = -\sum_{\none\neq S\subseteq[s]} g_A(S) \leq -s < 0.
$$

\begin{PROP} \label{tu-gnneg}
Let $\B=\{\beta_1,...,\beta_s\}\subseteq\{-1,0,+1\}^E$ be a set
of column vectors, let $U$ be the $E$-by-$s$ matrix with columns
$\{\beta_1,\ldots,\beta_{s}\}$, and let $A=(a_{ij})=U^\dagger U = \Gram(\B)$.
For each $i\in[s]$ let $C_i=\supp(\beta_i)$, and let $\C=\{C_1,...,C_s\}$.
If $U$ is TU then for all $\none\neq S\subseteq[s]$
we have $f_A(S) = \phi_\C(S)$ and $g_A(S) = \gamma_\C(S)$,
so that $U^\dagger U$ is $g$-nonnegative.
\end{PROP}
\begin{proof}
We use the notation $\beta_i(e)=U_{ei}$ for the entries of the matrix $U$.
We claim that for all $\none\neq S\subseteq[s]$,
$$
f_A(S) = \phi_\C(S).
$$
From this it follows by Inclusion/Exclusion that for all
$\none\neq S\subseteq[s]$,
$$
g_A(S) = \gamma_\C(S).
$$
The combinatorial meaning of $\gamma_\C$ then shows that
$A$ is $g$-nonnegative.

To prove the claim, consider any nonempty $S\subseteq[s]$.

If $S=\{i\}$ then
$$
f_A(\{i\}) = a_{ii} = \la \beta_i, \beta_{i} \ra = |C_i| = \phi_\C(\{i\}).
$$

If $S=\{i,j\}$ then consider any $\{e,f\}\subseteq C_i\cap C_j$.
Since $U$ is TU, the submatrix $Z$ of $U$ supported on
rows $e$ and $f$ and columns $i$ and $j$ has $\det Z\in\{-1,0,+1\}$.
All four entries of $Z$ are in $\{-1,+1\}$.
Computing the determinants of all possibilities one finds that
$Z$ has an even number of $-1$s, that $\det Z = 0$, and that $Z$ has rank one.
That is,
$$
\beta_i(e) \beta_j(e) = \beta_i(f) \beta_j(f).
$$
It follows that the function $e\mapsto \beta_i(e) \beta_j(e)$ is constant on
$C_i\cap C_j$, so that $|C_i\cap C_j|=|a_{ij}|$.  Equivalently,
for any $e\in C_i\cap C_j$,
$$
a_{ij} \cdot \beta_i(e) \beta_j(e) > 0. 
$$
(This is true even if $a_{ij}=0$, since then $C_i\cap C_j=\none$.)
Therefore,
$$
f_A(\{i,j\}) = |a_{ij}| = |C_i\cap C_j| = \phi_\C(\{i,j\}).
$$

It remains to consider the case that $|S|\geq 3$.

First, consider any $\{h,i,j\}\subseteq S$.  If $e\in C_h\cap C_i \cap C_j$
then from the above it follows that
$$
a_{hi}a_{ij}a_{jh} \cdot \beta_h(e)^2 \beta_i(e)^2 \beta_j(e)^2 > 0,
$$
and hence that $\{h,i,j\}$ is a positive triple for $A$.
Thus, if $S$ contains a negative or a null triple $\{h,i,j\}$ then 
$$
f_A(S) = 0 = |C_h\cap C_i\cap C_j|
= \left| \bigcap_{k\in S} C_k \right| = \phi_\C(S).
$$

Finally, consider the case that every triple contained in $S$ is positive.
We show that $f_A(S)=\phi_\C(S)$ by contradiction,
so suppose that there exists a set $S\subseteq[s]$ such that
$f_A(S)\neq\phi_\C(S)$.  Then there is such a set for which $S$ is minimal
according to set inclusion;  by the above observations, $|S|=t\geq 3$.
Replacing $\beta_{i}$ by $-\beta_{i}$ as necessary, we can assume that
$a_{ij}>0$ for all $\{i,j\}\subseteq S$.
(This is proved by induction on $t$; the base case $|t|=3$ and the
induction step both rely on the fact that every triple contained in $S$
is positive.)  Then, multiplying rows of 
$U$ by $-1$ as necessary, we can assume that $\beta_i(e)=1$ for all 
$i\in S$ and $e\in C_i$. 
Let $\{i,j\}\subset S$ be such that $a_{ij}$ is minimal.  Note that since
$A=\Gram(\B)$ and each $\beta_i\in\{-1,0,+1\}^E$, we have
$a_{ij} \leq \min \{a_{ii}, a_{jj}\}$ for all $\{i,j\}\subseteq[s]$.
Also note that for every $S'\subseteq S$ with $|S'|\geq 2$, we have
$f_{A}(S')\geq f_{A}(S) = a_{ij}$.
Now
$$
\phi_\C(S) = \left| \bigcap_{\ell\in S} C_\ell \right|
\leq |C_i\cap C_j| = a_{ij} = f_{A}(S).
$$
Since $S$ is a minimal set for which $f_{A}(S)\neq\phi_{\C}(S)$, 
it follows that $\phi_\C(S) < f_{A}(S)$, and that for every $h\in S$,
$$
\phi_\C(S\drop\{h\}) = f_A(S\drop\{h\}) \geq f_{A}(S) > \phi_\C(S).
$$
Therefore, for every $h\in S$ there is an element
$$
e_{h}\in \left(\bigcap_{\ell\in S\drop\{h\}} C_{\ell}\right) \drop C_{h}.
$$
These elements are pairwise distinct.
Let $Z$ be the submatrix of $U$ supported on columns $\{\beta_{i}:\ 
i\in S\}$ and rows $\{e_{h}:\ h\in S\}$.  By permuting rows and 
columns of $Z$ we can bring this into the form $J_{t}-I_{t}$, in which
$J_{t}$ is the $t$-by-$t$ all-ones matrix.  This is the adjacency 
matrix of the complete graph $K_{t}$, which has eigenvalues
$t-1$ of multiplicity $1$ and $-1$ of multiplicity $t-1$.  Therefore, since
$t\geq 3$, we see that
$$
\det Z = \pm \det(J_{t}-I_{t}) = \pm (t-1) \not\in \{-1,0,+1\}.
$$
This contradicts the hypothesis that $U$ is TU, showing that the
defective set $S\subseteq[s]$ is impossible.  This completes the proof.
\end{proof}

Let $A=(a_{ij})$ be an $s$-by-$s$ $g$-nonnegative matrix, and let
$k= -g_{A}(\none) \geq 0$.  Define a $k$-by-$s$ $\{0,1\}$-matrix $X(A)$
by saying that for each $\none\neq S\subseteq[s]$, exactly $g_{A}(S)$
rows of $X(A)$ are equal to the indicator row-vector of the subset 
$S\subseteq[s]$.  (Note that $X(A)$ has no zero rows.)  The
matrix $X(A)$ is defined only up to arbitrary permutation of the rows.
When $A$ is $g$-positive we usually permute the rows of $X(A)$ so that
the bottom $s$ rows form an identity submatrix $I_{s}$.

\begin{CORO} \label{tu-xa}
Let $U$ be a TU matrix, and let $A=U^{\dagger}U$ (which is
$g$-nonnegative).  Then the rows of $X(A)$ can be permuted so that
they are exactly the nonzero rows of $U^{\sharp}$.
\end{CORO}
\begin{proof}
Since $U$ is TU, we have $g_{A}(S)=\gamma_{\C}(S)$ for all $\none\neq 
S\subseteq [s]$, using the result and notation of Proposition \ref{tu-gnneg}.
Thus, for all $\none\neq S\subseteq [s]$, exactly $g_{A}(S)$ rows of $U$ have
support equal to the set $S$ of columns.  By definition, the same is true of
$X(A)$.  The matrix $U$ may also have some zero rows.
\end{proof}

A symmetric matrix $A$ is \emph{$g$-feasible} if there is a TU matrix $U$
such that $U^\dagger U = A$.  By Proposition \ref{tu-gnneg}, this implies that
$A$ is $g$-nonnegative.  Corollary \ref{tu-xa} and Camion's Lemma \ref{camion}
show that if such a matrix $U$ exists then it is unique (modulo deleting zero
rows, permuting the rows, and changing the signs of some rows and columns).
This is the uniqueness result at the heart of our proof of Theorem \ref{main}.

\subsection{Proof of Theorem 1.}

One last technical detail is required.

\begin{LMA} \label{tu-gram}
Let $U$ be an $m$-by-$s$ TU matrix containing $I_s$ as a submatrix.
Then every WU matrix $Q$ such that $Q^{\dagger}Q = U^\dagger U$ is TU.
\end{LMA}
\begin{proof}
We proceed by induction on $s$.  The basis of induction, $s=1$, is
trivial since in this case if $Q$ is WU then $Q$ is TU.

For the induction step we begin by showing that all  $(s-1)$-by-$(s-1)$
minors of  $Q$  are in  $\{-1,0,+1\}$.  Let $Z'$  be a nonsingular
$(s-1)$-by-$(s-1)$ submatrix of  $Q$.  Let  $Z$  be a nonsingular
$s$-by-$s$ submatrix of  $Q$  that contains $Z'$.  Then
$\det(Z)=\pm 1$, since $Q$ is WU, so that $F = Z^{-1}$ also has
$\det(F)=\pm 1$.  Now  $QF$  is  WU  and contains  $I_s$  as a submatrix,
so  $QF$  is  TU by Lemma \ref{wu2tu}.  Permuting this $I_s$ submatrix of
$QF$ to the bottom $s$ rows, the columns of  $QF$  are a fundamental
basis of a lattice $\Lambda(\N)$  for some regular matroid  $\N$.  Similarly,
the columns of  $U$  are a fundamental basis of a lattice  $\Lambda(\M)$
for some regular matroid  $\M$ (after permuting the $I_s$ submatrix of $U$ to
the bottom $s$ rows).  Since  $Q^{\dagger}Q = U^\dagger U$,
it follows that $(QF)^{\dagger}QF = (UF)^{\dagger}UF$.
Thus, the $i$-th column of $UF$  is the image of the $i$-th column of  
$QF$ (for each $i\in[s]$)  by means of an isometry from
$\Lambda(\N)$  to  $\Lambda(\M)$.  Since the columns of  $QF$  are simple
flows in  $\Lambda(\N)$  (by Lemma \ref{fact-e}), it follows from Lemma \ref{sf-iso} that
the columns of  $UF$  are simple flows in  $\Lambda(\M)$.  Thus, by
Proposition \ref{sf-metric}, the columns of  $UF$  are  $\{-1,0,+1\}$-valued.
Since  $U$  contains  $I_s$,  $UF$  contains $I_s F = F$ as a submatrix.
Thus, the entries of $F = Z^{-1} = \mathrm{adj}(Z)/\det(Z)$  are all in the set
$\{-1,0,+1\}$. Therefore $\det(Z')=\pm 1$, as required.

Now, for any  $i\in[s]$,  let  $Q_{i}$  be the submatrix of  $Q$  obtained
by deleting column  $i$  from  $Q$, and define  $U_{i}$  similarly.
Clearly  $U_{i}$  is  TU,  $U_{i}$  contains  $I_{s-1}$ as a submatrix,
and  $Q_{i}^{\dagger}Q_{i} = U_{i}^{\dagger} U_{i}$.  The previous paragraph
shows that each $Q_{i}$  is  WU.  Finally, the induction hypothesis shows that
each  $Q_{i}$  is  TU, and since  $Q$ is also  WU  it follows that  $Q$  is  TU.
This completes the induction step, and the proof.
\end{proof}

\begin{proof}[Proof of Theorem 1.]
We begin by proving that if $\M$ and $\N$ are regular matroids for 
which $\M_\bullet$ and $\N_\bullet$ are isomorphic, then $\Lambda(\M)$
and $\Lambda(\N)$ are isometric.
Let $\phi:E(\M_\bullet)\goesto E(\N_\bullet)$ be an isomorphism,
let $r$ be the rank of $\M_\bullet$ and let $k=|E(\M_\bullet)|$.
Let $B$ be any base of $\M_\bullet$ and let $\phi(B)$ be the
corresponding base of $\N_\bullet$.  Coordinatized by these bases,
both $\M_\bullet$ and $\N_\bullet$ are represented by the same $r$-by-$k$
matrix of the form $[I_r\ L]$ for some $r$-by-$s$ TU matrix $L$
(in which $s=k-r$).  Since $\M_{\bullet}$ has no co-loops, $L$ has no
zero rows. Let $\M$ have $p$ co-loops, and let $\N$ have $q$ co-loops.
Then $\M$ and $\N$ are represented by the matrices $M$ and $N$, 
respectively, in which
$$
M = \left[ \begin{array}{ccc}
I_{p} & O_{p\times r} & O_{p\times s} \\
O_{r\times p} & I_{r} & L
\end{array}\right]
\ \ \ \mathrm{and}\ \ \
N = \left[ \begin{array}{ccc}
I_{q} & O_{q\times r} & O_{q\times s} \\
O_{r\times q} & I_{r} & L
\end{array}\right].
$$
Here, $O_{a\times b}$ denotes the $a$-by-$b$ all-zero matrix.

As in Subsection 2.3, the lattices $\Lambda(\M)$ and $\Lambda(\N)$
have bases given by the columns of the matrices
$$
Q_{\M} = \left[ \begin{array}{c}
O_{p\times s} \\
-L \\
I_s
\end{array}\right]
\ \ \ \mathrm{and}\ \ \
Q_{\N} = \left[ \begin{array}{ccc}
O_{q\times s} \\
-L \\
I_s
\end{array}\right],
$$
respectively. One sees immediately that
$$
Q_{\M}^{\dagger} Q_{\M} = Q_{\N}^{\dagger} Q_{\N},
$$
and it follows that the lattices $\Lambda(\M)$ and $\Lambda(\N)$ are
isometric.

Conversely, assume that $\M$ and $\N$ are regular matroids and let
$\psi:\Lambda(\M)\goesto\Lambda(\N)$ be an isometry.  Let $s$ be the
rank of $\Lambda(\M)$ and $\Lambda(\N)$.  Let $|E(\M)|=m$ and 
$|E(\N)|=n$.

Let $\B=\B(\M,B)$ be a fundamental basis of $\Lambda(\M)$ 
coordinatized by a base $B$ of $\M$.  Let $U$ be an $m$-by-$s$ matrix
with the elements $\beta_{i}\in\B$ for $i\in[s]$ as columns.
Fix a TU matrix $N$ representing $\N$ over $\RR$, such that
$\Lambda(\N)=\ker(N)\cap\ZZ^{n}$.  Let $Q$ be the $n$-by-$s$ matrix
with the elements $\psi(\beta_{i})$ for $i\in[s]$ as columns.

Now $U$ is an $m$-by-$s$ TU matrix that contains $I_{s}$ as a 
submatrix, and since $\psi$ is an isometry it follows that 
$Q^{\dagger} Q = U^{\dagger} U$.  Since $\psi$
is an isometry and $\B$ is a basis for $\Lambda(\M)$, the columns of 
$Q$ form a basis for $\Lambda(\N)$.  From Lemma \ref{basis-wu} it follows
that $Q$ is WU, and then from Lemma \ref{tu-gram} it follows that $Q$ is TU.

Now both $U$ and $Q$ are TU matrices such that $A = U^{\dagger} U = 
Q^{\dagger} Q$.  By Corollary \ref{tu-xa}, the rows of $U$ and of $Q$ can
be permuted so that the nonzero rows of $U^{\sharp}$ and of $Q^{\sharp}$
both agree with $X=X(A)$.  Let $k=-g_{A}(\none)$ be the number of rows
of $X$, let $r = k-s$, let $p=m-k$, and let $q=n-k$.  We may assume that
the last $s$ rows of $X$ support an $I_{s}$ submatrix, so that 
$X = [K^{\dagger}\ I_{s}]^{\dagger}$
for some $r$-by-$s$ matrix $K$ with no zero rows.  Thus, the matrices
$U^{\sharp}$ and $Q^{\sharp}$ have the forms
$$
U^{\sharp} = \left[ \begin{array}{c}
O_{p\times s} \\
K \\
I_s
\end{array}\right]
\ \ \ \mathrm{and}\ \ \
Q^{\sharp} = \left[ \begin{array}{ccc}
O_{q\times s} \\
K \\
I_s
\end{array}\right].
$$
By Camion's Lemma \ref{camion}, there are diagonal matrices $H$ and $F$,
invertible over $\ZZ$, such that the submatrix in the last $k$ rows of
$Q'=HQF$ equals the the submatrix in the last $k$ rows of $U$.  The
columns of $Q'$ form a basis for $\Lambda(\N)$, and the matrices $U$
and $Q'$ have the forms
$$
U = \left[ \begin{array}{c}
O_{p\times s} \\
-L \\
I_s
\end{array}\right]
\ \ \ \mathrm{and}\ \ \
Q' = \left[ \begin{array}{ccc}
O_{q\times s} \\
-L \\
I_s
\end{array}\right]
$$
for some $r$-by-$s$ TU matrix $L$ with no zero rows.  Thus the regular 
matroids $\M$ and $\N$ are represented (over $\RR$) by the matrices
$$
M = \left[ \begin{array}{ccc}
I_{p} & O_{p\times r} & O_{p\times s} \\
O_{r\times p} & I_{r} & L
\end{array}\right]
\ \ \ \mathrm{and}\ \ \
N = \left[ \begin{array}{ccc}
I_{q} & O_{q\times r} & O_{q\times s} \\
O_{r\times q} & I_{r} & L
\end{array}\right],
$$
respectively.  From the forms of these representing matrices one sees that
$\M_{\bullet}$ and $\N_{\bullet}$ are both represented by $[I_{r}\ L]$,
and thus are isomorphic.
\end{proof}

\section{Concluding observations.}

\subsection{The lattice of integer cuts.}

Recall the lattice $\Gamma(\M)$ of integer cuts of a regular
matroid $\M$, defined in Subsection 2.3.   Theorem \ref{main} is equivalent
to each of the following two statements.  (We omit the trivial proofs by
duality).  For any matroid $\M$, let $\M^\circ$ denote the minor of $\M$
obtained by deleting all loops of $\M$.

\begin{CORO}
Let $\M$ and $\N$ be regular matroids.  Then
$\Gamma(\M)$ and $\Gamma(\N)$ are isometric if and only
if $\M^\circ$ and $\N^\circ$ are isomorphic.
\end{CORO}

\begin{CORO}
Let $\M$ and $\N$ be regular matroids.  Then
$\Lambda(\M)$ and $\Gamma(\N)$ are isometric if and only
if $\M_\bullet$ and $(\N^\circ)^* = (\N^*)_\bullet$ are isomorphic.
\end{CORO}

\subsection{Lattices in general.}

For convenience, a basis $\B$ of a lattice $\Lambda$ is said to be
$g$-nonnegative, $g$-positive, or $g$-feasible depending on whether
$\Gram(\B)$ has that property.

\begin{PROP}
Let $\Lambda$ be an integral lattice.  The following are equivalent.\\
\textup{(a)}\ $\Lambda$ has a $g$-feasible basis.\\
\textup{(b)}\ $\Lambda$ has a $g$-feasible and $g$-positive basis.\\
\textup{(c)}\ $\Lambda$ is isometric with $\Lambda(\M)$ for some regular
matroid $\M$.
\end{PROP}
\begin{proof}
For (a) implies (b):\ let $\B$ be a $g$-feasible basis for $\Lambda$.
Let $A=\Gram(\B)$, let $X=X(A)$, and let $U$ be a TU signing
of $X$ such that $U^\dagger U=A$.  Say that $U$ is an $m$-by-$s$ matrix.
Since $A$ has rank $s$, there is an invertible $s$-by-$s$ submatrix
$Z$ of $U$.  Since $U$ is TU, $\det(Z)=\pm 1$, so that $F=Z^{-1}$ is an
integer matrix and $\det(F)=\pm 1$ as well.  Now, $Q=UF$ is an 
$m$-by-$s$ WU matrix that contains $I_{s}$ as a submatrix, so by 
Lemma \ref{wu2tu}, $Q$ is TU.  The columns of $Q$ form a basis for $\Lambda$
(since $F$ is invertible over $\ZZ$), and $Q^{\dagger} Q$ is 
$g$-positive (by Proposition \ref{tu-gnneg}, and since $Q$ contains
$I_{s}$). Since $Q^{\dagger} Q$ is clearly $g$-feasible, this proves (b).

For (b) implies (c):\ if $\B$ is a $g$-feasible and $g$-positive basis of $\Lambda$
then $A=\Gram(\B)$ is $g$-positive and $X=X(A)$ has a TU signing $U$ such that
$U^\dagger U=A$.  The columns of $U$ form a basis $\B'$ of the lattice $\Lambda(\M)$
of some regular matroid.  Now
$$
\Gram(\B')=Q^\dagger Q = A = \Gram(\B),
$$
so that $\Lambda$ and $\Lambda(\M)$ are isometric.

Trivially (b) implies (a).  For (c) implies (b):\
assume that $\psi:\Lambda(\M)\goesto \Lambda$ is
an isometry, and let $B$ be any base of $\Lambda(\M)$.
Then $\B=\B(\M,B)$ is a $g$-feasible and $g$-positive basis of 
$\Lambda(\M)$, so that $\psi(\B)$ is a $g$-feasible and $g$-positive
basis of $\Lambda$.
\end{proof}

\subsection{Some examples.}

\begin{EG} \label{eg1}
A $g$-positive matrix that is not $g$-feasible.
\emph{
The matrix $A$ shown below is $g$-positive, with $X=X(A)$ as shown.
$$
A = 
\left[\begin{array}{cccc}
3 & 1 & 1 & 2 \\
1 & 3 & 1 & 2 \\
1 & 1 & 3 & 2 \\
2 & 2 & 2 & 5
\end{array}\right]
\hspace{1cm}
X = 
\left[\begin{array}{cccc}
1 & 1 & 1 & 1 \\
1 & 0 & 0 & 1 \\
0 & 1 & 0 & 1 \\
0 & 0 & 1 & 1 \\
1 & 0 & 0 & 0 \\
0 & 1 & 0 & 0 \\
0 & 0 & 1 & 0 \\
0 & 0 & 0 & 1
\end{array}\right]
$$
Theorem 13.1.3 of \cite{Ox} shows that $X$ does not have a TU signing
(by pivotting on the top-right entry.)  Thus, $A$ is not 
$g$-feasible. 
}\end{EG}

\begin{EG} \label{eg2}
A $g$-nonnegative matrix $A$ such that $X(A)$ has a TU signing,
but $A$ is not $g$-feasible.
\emph{
The matrix $A$ shown below is $g$-nonnegative, with $X=X(A)$ as shown.
$$
A = 
\left[\begin{array}{cccc}
2  & 1 & 0 & -1 \\
1  & 2 & 1 & 0  \\
0  & 1 & 2 & 1  \\
-1 & 0 & 1 & 2
\end{array}\right]
\hspace{1cm}
X = 
\left[\begin{array}{cccc}
1 & 0 & 0 & 1 \\
1 & 1 & 0 & 0 \\
0 & 1 & 1 & 0 \\
0 & 0 & 1 & 1 
\end{array}\right]
$$
One checks that $X$ itself is TU.  By Camion's Lemma \ref{camion}, any 
TU signing $U$ of $X$ is obtained from $X$ by multiplying some rows 
and columns of $X$ by $-1$.  For any such matrix $U$, the Gram matrix
$U^{\dagger}U$ is obtained from $X^{\dagger}X$ by multiplying some 
rows and the same columns by $-1$.  But $X^{\dagger}X = A^{\sharp}$,
and $A$ cannot be obtained from $A^{\sharp}$ by means of this
operation.  Thus, there is no TU signing $Q$ of $X$ such that 
$Q^{\dagger}Q = A$.  (A $g$-positive example of this is $A+I_{4}$.)
}\end{EG}

\begin{EG} \label{eg3}
A matrix $Q$ such that $Q^{\dagger}Q$ is $g$-positive and
$g$-feasible, but $Q$ is not WU.
\emph{
This example relates to the hypotheses of Lemma \ref{tu-gram}.
Clearly $Q=[2]$ is not WU.   The matrix $A=Q^\dagger Q =[4]$ is
$g$-positive with $X=X(A)=[1\ 1\ 1\ 1]^\dagger$.
Clearly $X$ is TU with $X^{\dagger}X=A$, so $A$ is $g$-feasible.
}\end{EG}

\begin{EG} \label{bcc}
The body-centered cubic lattice is $\Lambda(K_4)\simeq\Gamma(K_4)$.
\emph{
To see this, let the cycles of length four in $K_4$ be $C_1$, $C_2$, and
$C_3$, and let $\alpha_i$ be a simple flow supported on $C_i$ for each
$i\in\{1,2,3\}$.  Now $\{\alpha_1, \alpha_2, \alpha_3\}$ spans a sublattice
$\Pi$ of $\Lambda(K_4)$, and has Gram matrix $4I_3$.  Thus, $\Pi$ is a
cubical lattice with minimum length $2$.  Now,  $\alpha_1+\alpha_2+\alpha_3
= 2\beta$ for some simple flow $\beta\in\Lambda(K_4)$ supported on a
three-cycle.  In fact $\Lambda(K_4)$ is the disjoint union of $\Pi$ and $\Pi+\delta$, proving the claim.  The lattices $\Lambda(K_n)$ are discussed
on pages 194--196 of
\cite{BHN}.
}\end{EG}

\begin{EG} \label{a-root}
The root lattice $\mathsf{A}_n$ is $\Lambda(\U_{1,n+1})\simeq\Gamma(\U_{n,n+1})$.
(The face-centered cubic lattice is $\mathsf{A}_3$.)
\emph{
To see this, for each $i\in[n+1]$ let $\e_i$ be the coordinate column
vector of length $n+1$ with all entries $0$ except for a $1$ in row $i$.
The root lattice $\mathsf{A}_n$ has as a basis the vectors
$s_i= \e_{i+1} - \e_{i}$ for all $i\in[n]$.  Since $\U_{1,n+1}$ is represented
by the all-ones matrix with one row and $n+1$ columns, it is easy to see that
$\{s_1,...,s_n\}$ is a basis for $\Lambda(\U_{1,n+1})$ as well.
The argument on page 194 of \cite{BHN} shows that these are the only root
lattices of the form $\Lambda(\M)$ for some regular matroid.
}\end{EG}

\subsection{Sixth-root-of-unity matroids.}

Let $\omega = \mathrm{e}^{\mathrm{i}\pi/3}$ be a primitive sixth-root of unity,
and let
$$
\mathbb{E} = \{z\in\CC:\ z= a + b \omega\
\text{for some $a,b\in\ZZ$} \}
$$
be the ring (in fact a PID) of \emph{Eisenstein integers}.
A \emph{sixth-root-of-unity matrix} ($\sqrt[6]{1}$ matrix, for short)
is a matrix with entries in $\CC$ such that every square submatrix has
determinant $d$ such that either $d=0$ or $d^6=1$.  A \emph{sixth-root-of-unity
matroid} ($\sqrt[6]{1}$ matroid, for short) is one which can be represented
over $\CC$ by a $\sqrt[6]{1}$ matrix.  Clearly, regular matroids are $\sqrt[6]{1}$.
Lemma 5.8 of \cite{Whi} gives sufficient conditions for a $\sqrt[6]{1}$  
matroid to be uniquely representable over $\CC$ (by a $\sqrt[6]{1}$  
matrix).

For example, $\U_{2,4}$ is not a binary matroid (hence not regular) but it
is represented over $\CC$ by the matrix
$$
\left[\begin{array}{cccc}
1 & 0 & 1 & 1 \\
0 & 1 & 1 & \omega
\end{array}\right].
$$

The vector space of real flows of $\U_{2,4}$ (relative to this representation)
is the real span of the column vector $[-1\ -1\ 1\ 0]^\dagger$.  The vector
space of complex flows of $\U_{2,4}$ is the complex span of both
$[-1\ -1\ 1\ 0]^\dagger$ and $[-1\ -\overline{\omega}\  0\ 1]^\dagger$.
The analogue of the lattice of integer flows for a matroid $(\M,E)$ 
represented over $\CC$ by a $\sqrt[6]{1}$ matrix $M$ is
the \emph{lattice of Eisenstein flows}
$$
\Lambda_\EE(\M) = \ker_\CC(M)\cap \EE^E.
$$
The inner product on $\Lambda_{\EE}(\M)$ is induced by the
Hermitian inner product on $\CC^{E}$.

Such a lattice is not just an abelian group, but even an $\EE$-module.
This allows a stronger version of isometry:  $\psi:\Lambda\goesto\Lambda'$
is an \emph{$\EE$-isometry} if it is a bijection such that both 
$\psi$ and $\psi^{-1}$ are $\EE$-module homomorphisms that preserve the 
inner products on the lattices.  Clearly an $\EE$-isometry is an 
isometry in the usual sense.

If $M$ and $M'$ are $\sqrt[6]{1}$ matrices representing the same
matroid $\M$, then $\ker_\CC(M)\cap \EE^E$ and $\ker_\CC(M')\cap 
\EE^E$ are $\EE$-isometric if and only if $M$ and $M'$ are 
equivalent representations of $\M$.  (This follows easily from the
definition of representation equivalence.)
Thus, $\M$ is uniquely representable by a $\sqrt[6]{1}$ matrix if and 
only if the $\EE$-isometry class of $\Lambda_\EE(\M)$ is independent of the
representing matrix $M$.  It is not too difficult to see that for a regular
matroid $\M$,
$$
\Lambda_\EE(\M) = \Lambda(\M) \otimes \EE,
$$
but, as the example of $\U_{2,4}$ shows, this does not hold for all
$\sqrt[6]{1}$ matroids.  (In fact, this equality holds if and only if
$\M$ is regular, since if it holds then $\Lambda_\EE(\M)$ has an $\EE$-basis
$\B$ that is also a $\ZZ$-basis of $\Lambda(\M)$.  Thus, the matrix
$Q$ formed from the column vectors in $\B$ is $\sqrt[6]{1}$ and real, hence TU.
Therefore $\M^*$ and hence $\M$ are regular.)

The two-sum $\U_{2,4}\oplus_{2}\U_{2,4}$ has two inequivalent 
representations by $\sqrt[6]{1}$ matrices, and these yield Eisenstein
flow lattices that have bases with Gram matrices
$$
\left[
\begin{array}{rrr}
3 & 1+\omega & 1 + \omega \\
1+\overline{\omega} & 4 & 4 - \sigma \\
1+\overline{\omega} & 4 - \overline{\sigma} & 4
\end{array}
\right]
$$
in which $\omega,\sigma\in\CC$ are primitive sixth-roots of unity.
The two cases $\sigma=\omega$ and $\sigma = \overline{\omega}$ yield lattices
which are not $\EE$-isometric, but seem to be isometric.

\begin{CONJ} \label{repn-iso}
If $M$ and $M'$ are sixth-root-of-unity matrices representing the
same matroid $\M$, then $\ker_\CC(M)\cap \EE^E$ and $\ker_\CC(M')\cap 
\EE^E$ are isometric.
\end{CONJ}

\begin{CONJ} \label{sixth-root}
Let $\M$ and $\N$ be sixth-root-of-unity matroids.
Then  $\Lambda_\EE(\M)$ and $\Lambda_\EE(\N)$ are
isometric if and only if $\M_{\bullet}$ and $\N_{\bullet}$ are isomorphic.
\end{CONJ}

One could perhaps adopt a strategy similar to the one we used to prove Theorem
\ref{main} for regular matroids.  The Gram matrix of a basis of $\Lambda_\EE(\M)$
is in general complex Hermitian with entries in $\EE$.  If one can identify
a metric characterization of simple flows, and an appropriate generalization of
$g$-feasible matrices, then much of our argument could carry over.



\begin{thebibliography}{ABC}

\bibitem{BHN}
R.~Bacher, P.~de~la~Harpe, and T.~Nagnibeda,
\emph{The lattice of integral flows and the lattice of integral cuts
on a finite graph},
Bull. Soc. math. France \textbf{125} (1997), 167--198.

\bibitem{Bi}
N.L.~Biggs,
\emph{Algebraic potential theory on graphs},
Bull. London Math. Soc. \textbf{29} (1997), 641--682.

\bibitem{Ox}
J.G.~Oxley,
``Matroid Theory,''
Oxford U.P., New York, 1992.

\bibitem{Tut1}
W.T.~Tutte,
\emph{A class of abelian groups},
Canadian J. Math. \textbf{8} (1956), 13--28.\\
also in ``Selected Papers of W.T. Tutte, Vol. I,''
Charles Babbage Research Center, St. Pierre, 1979.

\bibitem{Tut2}
W.T.~Tutte,
\emph{Lectures on matroids},
J. Res. Nat. Bur. Stand. Sect. B \textbf{69} (1965), 1--47. 
also in ``Selected Papers of W.T. Tutte, Vol. II,''
Charles Babbage Research Center, St. Pierre, 1979.

\bibitem{Wh}
H.~Whitney,
\emph{$2$-Isomorphic Graphs},
Amer. J. Math. \textbf{55} (1933), 245--254.

\bibitem{Whi}
G.~Whittle,
\emph{On matroids representable over $\mathrm{GF}(3)$ and other fields},
Trans. Amer. Math. Soc. \textbf{349} (1997), 579-603.

\end{thebibliography}
\end{document}